\documentclass[final,hidelinks,onefignum,onetabnum]{siamart220329}

\usepackage[margin=1in, left=1.5in, right=1.5in]{geometry} 

\usepackage{lipsum}
\usepackage{amsfonts,amssymb,amsmath}
\usepackage{physics}
\usepackage{microtype}
\usepackage{graphicx}
\usepackage{epstopdf}
\usepackage{algorithmic}
\usepackage{pgfplots}\pgfplotsset{compat=1.9}
\usepackage{array} 
\usepackage[caption=false]{subfig} 
\ifpdf
  \DeclareGraphicsExtensions{.eps,.pdf,.png,.jpg}
\else
  \DeclareGraphicsExtensions{.eps}
\fi
\usepackage{graphicx,epstopdf} 
\usepackage[caption=false]{subfig} 
\usepackage{enumitem}
\setlist[enumerate]{leftmargin=.5in}
\setlist[itemize]{leftmargin=.5in}

\newsiamthm{remark}{Remark}
\newsiamthm{assumptions}{Assumption}

\headers{FISTA Converges Linearly for PnP and RED}{A. Sinha and K. N. Chaudhury}

\title{FISTA Iterates Converge Linearly for Denoiser-Driven Regularization\thanks{
    \textbf{Funding:} A.~Sinha was supported by PMRF Fellowship TF/PMRF-22-5534 and K.~N.~Chaudhury was supported by grants CRG/2020/000527 and
STR/2021/000011 from the Government of India.}}

\author{Arghya Sinha\thanks{Department of Computational and Data Sciences, Indian Institute of Science, Bangalore
    (\email{arghyasinha@iisc.ac.in}).}
  \and Kunal N. Chaudhury\thanks{Department of Electrical Engineering, Indian Institute of Science, Bangalore
    (\email{kunal@iisc.ac.in}).}
}

\usepackage{amsopn}
\DeclareMathOperator{\diag}{diag}

\newcommand{\Sy}{\mathbb{S}}
\renewcommand{\Re}{\mathbb{R}}

\def\A{\mathbf{A}}
\def\B{\mathbf{B}}

\def\D{\mathbf{D}}

\def\G{\mathbf{G}}
\def\H{\mathbf{H}}
\def\I{\mathbf{I}}

\def\K{\mathbf{K}}
\def\M{\mathbf{M}}

\def\P{\mathbf{P}}
\def\Q{\mathbf{Q}}
\def\R{\mathbf{R}}
\def\S{\mathbf{S}}
\def\T{\mathtt{T}}

\def\W{\mathbf{W}}

\def\b{\boldsymbol{b}}

\def\e{\boldsymbol{e}}

\def\q{\boldsymbol{q}}
\def\r{\boldsymbol{r}}
\def\s{\boldsymbol{s}}
\def\u{\boldsymbol{u}}
\def\v{\boldsymbol{v}}

\def\x{\boldsymbol{x}}
\def\y{\boldsymbol{y}}
\def\z{\boldsymbol{z}}
\def\ze{\mathbf{0}}
\def\one{\mathbf{1}}

\newcommand{\fix}{\mathrm{fix}}
\newcommand{\prox}{\mathrm{prox}}
\newcommand{\Span}{\mathrm{span}}

\newcommand{\mynegspace}{\hspace{-0.12em}}
\newcommand{\lvvvert}{\rvert\mynegspace\rvert\mynegspace\rvert}
\newcommand{\rvvvert}{\rvert\mynegspace\rvert\mynegspace\rvert}

\usepackage[normalem]{ulem}
\ifpdf
\hypersetup{
  pdftitle={PnPFISTA},
  pdfauthor={A. Sinha and K. N. Chaudhury}
}
\fi

\usepackage{microtype}
\usepackage{palatino}

\begin{document}

\maketitle

\begin{abstract}
The effectiveness of denoising-driven regularization for image reconstruction has been widely recognized. Two prominent algorithms in this area are Plug-and-Play (\texttt{PnP}) and Regularization-by-Denoising (\texttt{RED}). We consider two specific algorithms \texttt{PnP-FISTA} and \texttt{RED-APG}, where regularization is performed by replacing the proximal operator in the \texttt{FISTA} algorithm with a powerful denoiser. The iterate convergence of \texttt{FISTA} is known to be challenging with no universal guarantees. Yet, we show that for linear inverse problems and a class of linear denoisers, global linear convergence of the iterates of \texttt{PnP-FISTA} and \texttt{RED-APG} can be established through simple spectral analysis.
\end{abstract}


  

  \section{Introduction}

Over the past decade, there has been a growing interest in applying denoisers for iterative reconstruction. This trend was triggered by the seminal works \cite{romano2017little,sreehari2016plug}, where the idea of using off-the-shelf denoisers for regularization was formalized. Subsequent studies have revealed that trained denoisers produce results that are competitive with end-to-end deep learning methods. Notably, this approach allows us to combine robust denoising priors with model-based inversion for variational regularization. 
 
{The effectiveness of denoiser-driven regularization has motivated researchers to investigate its convergence properties~\cite{cohen_regularization_2021,gavaskarPlugandPlayRegularizationUsing2021a,gavaskar2020plug,goujon_learning_2024,hauptmann_convergent_2024,hurault2022gradient,liu_study_2021,nair2024averaged,pesquet_learning_2021,reehorst2018regularization,ryu2019plug,xu_provable_2020}. We refer the reader to the recent article~\cite{elad_image_2023} for an excellent survey on denoiser-driven regularization.}The convergence theory in this area is primarily motivated by convex optimization and monotone operator theory \cite{bauschke2019convex,ryu_scaled_2022}. Indeed, a common theme in many of these papers is to guarantee convergence by designing an appropriate regularizer \cite{cohen_regularization_2021,hurault2022gradient} or by restricting the denoiser (during training or by design) to a class of nonexpansive operators \cite{hertrich_convolutional_2021,pesquet_learning_2021}. While the current focus is on trained denoisers, it is also known that their black-box nature can lead to unpredictable behavior; e.g., {see~\Cref{fig:divergence}, \cite[Fig.~6]{nair2024averaged}, and \cite[Fig.~1]{terris2024equivariant}}. Moreover, the convergence guarantees for these denoisers are often based on assumptions that are difficult to verify in practice. In this work, we demonstrate that using a class of data-driven linear denoisers, which are not state-of-the-art but still give good reconstructions \cite{nair2022plug,nair2019hyperspectral,sreehari2016plug}, we can get strong convergence guarantees under fully verifiable assumptions.

\subsection{PnP and RED}
We consider two popular regularization methods, Plug-and-Play (\texttt{PnP}) and Regularization-by-Denoising (\texttt{RED}). In particular, we revisit the \texttt{PnP-FISTA} and \texttt{RED-APG} {(\texttt{RED}–Accelerated Proximal Gradient)} algorithms \cite{gavaskarPlugandPlayRegularizationUsing2021a,reehorst2018regularization} that are modeled on the classical \texttt{FISTA} algorithm {(Fast Iterative-Shrinkage Thresholding Algorithm)}. This is an {accelerated} first-order method for solving convex optimization problems of the form
  \begin{equation}
    \label{fplusg}
    \min_{\x \in \Re^n} \ f(\x) + g(\x),
  \end{equation}
where $f$ is smooth and $g$ is nonsmooth but proximable \cite{beck_fast_2009}. In image reconstruction, $f$ is a model-based loss function, and $g$ is an image regularizer. Starting with an initialization $\y_1=\x_0$ and $t_1=1$, \texttt{FISTA} generates a sequence of images $\{\x_k\}_{k \geqslant 1}$ using the updates:
  \begin{subequations}
    \begin{align}
      \x_{k}  &= \prox_{\gamma g} \big(\y_k - \gamma \nabla \! f(\y_k) \big), \label{xupdate} \\
      t_{k+1}  &= \frac{1}{2} \left(1+\sqrt{ 1+4t_k^2} \right),  \label{tk} \\
      \y_{k+1} &= \x_{k} + \frac{t_k-1}{t_{k+1}} \big(\x_{k}-\x_{k-1} \big),
    \end{align}
  \end{subequations}
  where $\prox_{\gamma g}$ is the proximal operator of $\gamma g$ and $\gamma >0$ is the step size. For sufficiently small $\gamma$ and appropriate technical assumptions on $f$ and $g$, it was shown by Beck and Teboulle that the objective $f(\x_k) + g(\x_k)$ converges sublinearly to the minimum of  \cref{fplusg}.

In \texttt{PnP-FISTA}, the proximal operator (which effectively performs image denoising) is substituted by a powerful denoiser, i.e., the update in \cref{xupdate} is performed as
  \begin{equation}
    \label{eq:pnpista}
    \x_{k} = \W  \big(\y_k - \gamma \nabla \! f(\y_k) \big),
  \end{equation}
  where $\W$ is a denoising operator. In this note, we will work with a particular class of linear denoisers. This includes kernel denoisers such as \texttt{NLM} \cite{buades2005non,sreehari2016plug}, \texttt{LARK}, and \texttt{GLIDE} \cite{milanfar2013symmetrizing}. This is the simplest \texttt{PnP} model that works well in practice \cite{gavaskarPlugandPlayRegularizationUsing2021a,gavaskar2023PnPCS,sreehari2016plug} and yet is analytically tractable. 
  {We wish to emphasize that this class of denoisers is not strictly ``linear" like the simple Gaussian or the box filter. Rather, the action of the denoiser is given by  $\x \mapsto \W(\x_g)\, \x$ which acts linearly on the input image $\x$ and depends non-linearly on a guide image $\x_g$; the latter serves as a surrogate for the ground truth image. In fact, it is precisely this data-driven nature, referred to as ``pseudo-linearity'' \cite{milanfar2013tour}, that makes them more effective than classical linear denoisers. In iterative image reconstruction, the initial guide image serves as a rough estimate that evolves over the first few iterations as per the updates in \cref{eq:pnpista}. To achieve high-quality reconstructions \cite{gavaskarPlugandPlayRegularizationUsing2021a,sreehari2016plug}, a common practice is to fix the guide image after a certain number of iterations. Thus, as far as the question of convergence is concerned, we can assume the denoiser to be a purely linear operator $\x \mapsto \W \x$. While the choice of guide determines the reconstruction quality, it has no direct bearing on the convergence problem.}

  In \texttt{PnP-FISTA}, we treat the loss $f$ as the smooth component of the objective function. On the other hand, in \texttt{RED-APG}, the accelerated variant of \texttt{RED}, we work with the first-order approximation of the regularizer $g$ \cite[Sec.~V]{reehorst2018regularization}. More precisely, starting with  $t_0=1$ and $\v_0 \in \Re^n$, the updates in \texttt{RED-APG} are given by
  \begin{subequations}
    \begin{align}
      \x_{k} &= \prox_{(\lambda L)^{-1} \, f } (\v_{k-1}), \\
      t_{k}  &= \frac{1}{2} \left(1+\sqrt{ 1+4t_{k-1}^2} \right), \\
      \y_{k} &= \x_{k} +  \frac{t_{k-1}-1}{t_k}  \, (\x_{k}-\x_{k-1}), \\
      \v_{k} &= \theta \, \W(\y_k)+ (1-\theta) \, \y_k,
    \end{align}
  \end{subequations}
  where $\lambda > 0, \, L \geqslant 1$ are the internal parameters and $\theta=1/L$.

\subsection{Relation to prior work}
Although \texttt{PnP} was originally conceived from an algorithmic point-of-view, the iterations in \texttt{PnP-FISTA} can nonetheless be seen as solving an optimization problem of the form \cref{fplusg}. As a result, we can establish objective convergence of \texttt{PnP-FISTA} and \texttt{RED-APG} using existing results \cite{beck2017book,beck_fast_2009}. In this work, we analyze their iterate convergence. This was motivated by a recent work \cite{ACK2023-contractivity}, where we were able to establish global linear convergence of \texttt{PnP-ISTA}, which is modeled on the \texttt{ISTA} algorithm for minimizing \cref{fplusg}. It is known that \texttt{FISTA} is faster than its non-accelerated counterpart \texttt{ISTA} in terms of objective convergence \cite{beck_fast_2009}; however, there is no direct relation in terms of iterate convergence (see \Cref{fig:plot}). In fact, while we can guarantee iterate convergence of \texttt{ISTA} using the concept of Mann iterations \cite{beck2017book}, establishing iterate convergence of \texttt{FISTA} is challenging with no universal convergence guarantees. In this direction, Chambolle and Dossal \cite{chambolle2015convergence} have shown that a different choice of momentum parameters can guarantee iterate convergence. Lorenz and Pock obtained similar results \cite{lorenz2015inertial}, although neither work established the convergence rate. On the other hand, asymptotic linear convergence of \texttt{FISTA} was established by Tao et al. \cite{tao2016local} and Johnstone and Moulin \cite{johnstone2015lyapunov} for sparse minimization such as LASSO. Similar to our work, spectral analysis of the underlying operators comes up in \cite{tao2016local}; however, their overall analysis is much more intricate due to the nonlinearity of the denoising operator in LASSO. Notably, the results mentioned above do not apply to  \texttt{PnP-FISTA} and \texttt{RED-APG}. In fact, to the best of our knowledge, convergence analysis of these algorithms has not been reported in the literature. 
  
  {The present contribution lies in establishing global linear convergence of the \texttt{PnP-FISTA} and \texttt{RED-APG} iterates in the context of linear inverse problems and data-driven linear denoisers. In this regard, we wish to comment on prior convergence results for linear denoisers~\cite{gavaskarPlugandPlayRegularizationUsing2021a,gavaskar2020plug,liu_study_2021,sreehari2016plug}. In \cite{gavaskarPlugandPlayRegularizationUsing2021a,sreehari2016plug}, it was demonstrated that \texttt{PnP} with kernel denoisers can be formulated as a convex optimization problem of the form \cref{fplusg}; objective convergence of \texttt{PnP-FISTA} and \texttt{PnP-ADMM} was subsequently established using standard optimization theory. Iterate convergence of \texttt{PnP-ISTA} and \texttt{PnP-ADMM} (but not \texttt{PnP-FISTA}) follows from standard results \cite{bauschke2019convex}. However, since the smooth part $f$ in \cref{fplusg} is typically not strongly convex, it is non-trivial to establish linear convergence. Yet, it was shown in \cite{gavaskar2020plug,liu_study_2021} that linear convergence can be established for \texttt{PnP-ISTA}. However, the analysis in these papers is tailored to the inpainting problem and relies on assumptions that are difficult to interpret and validate. More recently, we developed a convergence theory for \texttt{PnP-ISTA} and \texttt{PnP-ADMM} in \cite{ACK2023-contractivity} based on practically verifiable assumptions. A key difference with  \cite{ACK2023-contractivity} is that we work with the spectral radius instead of controlling a specific norm.}

\subsection{Notations} We use $\ze$ and $\one$ to denote the all-zeros and the all-ones vectors in $\Re^n$, the dimension $n$ should be clear from the context. We use $\ze_n$ and  $\I_n$ for the zero and identity operators on $\Re^{n}$. The space of symmetric linear operators on $ \Re^{n}$ is denoted by $\Sy^n$. We use $\fix (\W)$ for the fixed points of $\W$, $\ker(\A)$ for the  kernel of $\A$,  $\sigma(\B)$ and $\rho(\B)$ for the spectrum and the spectral radius of $\B$, $\mathrm{diag}(\u)$ for the diagonal matrix with $\u$ as the diagonal elements, $\A \sim \B$ to mean that $\A$ and $\B$ are similar, $\Span(\u)$ for the space spanned by vector $\u$, and $\| \cdot \|_1$, $\| \cdot \|_2$ and $\| \cdot \|_\infty$ to denote the $\ell_1, \ell_2$ and $\ell_\infty$ norms. {A matrix $\W$ is said to be stochastic if $\W$ is nonnegative and each of its rows sum to $1$. In addition, $\W$ is said to be irreducible if the directed graph associated with $\W$ is strongly connected \cite{Horn_Johnson_2012}.}
  
\subsection{Organization} We discuss kernel denoisers and the forward model in Section~\ref{sec:kd}, the main results are derived in Section~\ref{sec:analysis}, and we give numerical simulations in Section~\ref{sec:exp}.

\begin{figure}[t]
    \centering
    \subfloat[DRUNet]{
    \includegraphics[width=0.45\linewidth]{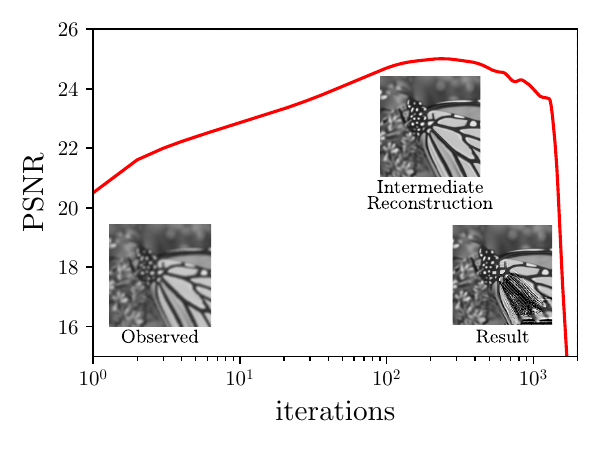}}\hfill
    \subfloat[DnCNN]{
    \includegraphics[width=0.45\linewidth]{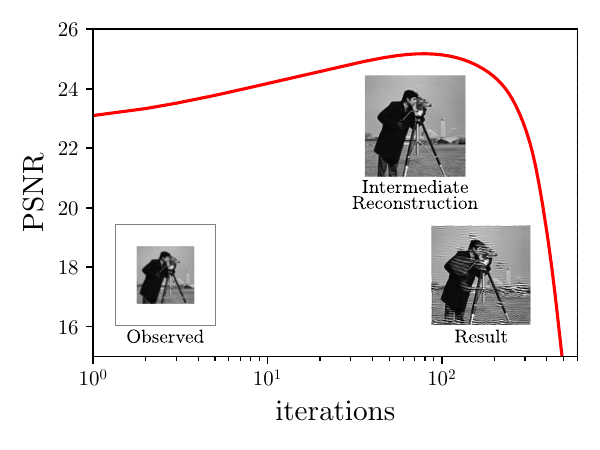}}\hfill
    \caption{
{An example demonstrating the divergence of PnP-ISTA with two commonly used deep denoisers, DRUNet~\cite{zhang2021plug} and DnCNN~\cite{dncnn}. The result in (a) is for Gaussian deblurring ($9 \times 9$ kernel with a standard deviation of $4$), and that in (b) for superresolution ($25 \times 25$ Gaussian blur with a standard deviation of $1.6$ and  $2 \times$ downsampling). In both examples, we observe that the PSNR initially improves but then declines after a certain point as the iterates diverge. Similar results have been reported in~\cite{terris2024equivariant}.}
 }
    \label{fig:divergence}
\end{figure}

  \section{{Background}}
  \label{sec:kd}
We consider the problem of recovering an unknown image $\boldsymbol{\xi} \in \Re^n$ from linear measurements
  $\b = \A \boldsymbol{\xi}  + \e$,
  where $\A \in \Re^{m \times n}$ is the forward operator, $\b \in \Re^m$ is the observed image, and $\e \in \Re^m$ is white Gaussian noise. In particular, we will use the standard loss function
  \begin{equation}
    \label{qloss}
    f(\x) = \frac{1}{2} \| \A \x - \b \|_2^2.
  \end{equation}
  
  As mentioned earlier, we will abstractly work with a linear denoiser $\W \in\Re^{n\times n}$ that transforms a noisy image $\x$ into $\W \x$, where $\x$ is a vectorized representation of the image and $n$ is the number of pixels. We will initially perform the analysis for a symmetric $\W$. Later, in \Cref{subsection:nonsymm}, we will extend the results to nonsymmetric denoisers. We do not require the precise definition of $\W$ but rather consider the following property.
  \begin{assumptions}
  \label{RNP}
      $\sigma(\W)\subset [0,1]$ and 
      $\ker(\A) \, \cap \, \fix(\W) = \{ \ze \}.$
  \end{assumptions}

   We will subsequently show that the above assumption holds for a specific class of denoisers called kernel denoisers \cite{milanfar2013tour}, which includes \texttt{NLM} \cite{buades2005non} and its symmetric variants \cite{milanfar2013symmetrizing,sreehari2016plug}. Broadly speaking, a kernel denoiser is obtained by considering a kernel function and using it to measure the affinity between pairs of pixels. This gives us a kernel operator $\K \in \Re^{n \times n}$, where $\K_{ij}$ encodes the affinity between pixels $i$ and $j$. As mentioned earlier, $\K$ is computed from a surrogate of the ground truth, called the guide image. More specifically, the kernel in \texttt{NLM} has the form 
  \begin{equation}
  \label{kernel}
\K_{ij} = \varphi\big(\zeta(i)-\zeta(j)\big) \, h(i-j),
\end{equation}
  where $ \varphi$ is a multivariate Gaussian, $\zeta(i)$ is a patch around pixel $i$ extracted from the guide image, and $h$ is a symmetric box or hat function supported on a window around the origin. The denoiser is obtained by normalizing $\K$:
  \begin{equation}
    \label{NLM}
    \W_{\texttt{NLM}}=\D^{-1}\, \K, \quad \D := \mathrm{diag}( \K \,\one).
  \end{equation}
It is evident that \cref{NLM} need not be symmetric. We can construct a symmetric variant of \cref{NLM} using the protocol for \texttt{DSG-NLM}~\cite{sreehari2016plug}:
\begin{equation}
\label{DSG-NLM}
   \W_{\texttt{SYM}} = \frac{1}{\|\hat{\one}\|_\infty} \D^{-\frac{1}{2}}\K\D^{-\frac{1}{2}} + \diag \left(\one - \frac{1}{\|\hat{\one}\|_\infty} \, \hat{\one} \right), \quad \hat{\one} :=  \D^{-\frac{1}{2}}\K\D^{-\frac{1}{2}}  \one,
\end{equation}
where we use the subscript $\texttt{SYM}$ to distinguish it from \cref{NLM}. We refer the reader to \cite{milanfar2013tour} for more details on such denoisers.

Kernel denoisers come with a rich set of mathematical properties. We abstract out the specific properties required for our analysis.

\begin{proposition}
\label{assumpK}
\label{propertiesW}
Suppose $\K$ is symmetric positive semidefinite, irreducible, and nonnegative with $\K_{ii} > 0$ for $1 \leqslant i \leqslant n$. Then \cref{NLM} and \cref{DSG-NLM} are well-defined, stochastic, and irreducible, and their eigenvalues are in $[0,1]$.
\end{proposition}

In particular, \cref{kernel} satisfies the assumptions in \Cref{propertiesW}. The details can be found in \cite[Sec.~IV]{sreehari2016plug}, which also has the proof of the assertions in \Cref{propertiesW}.

We next verify \cref{RNP} for denoisers \cref{NLM} and \cref{DSG-NLM}, when $\A$ corresponds to inpainting, deblurring, or superresolution. We assume that at least one pixel is sampled for inpainting, and the lowpass blur is nonnegative (and nonzero) for deblurring and superresolution; these technical conditions are typically met in practice \cite{bouman2022foundations}.
\begin{proposition}
    \label{verification-of-rnp}
    Suppose $\W$ is \cref{NLM} or \cref{DSG-NLM}, and the forward operator $\A$ corresponds to inpainting, deblurring, or superresolution. Then \cref{RNP} holds.
\end{proposition}
\begin{proof}
    From \Cref{propertiesW}, we have $\rho(\W_{\texttt{NLM}})=1$ and $\one \in \fix(\W_{\texttt{NLM}})$. Moreover, since $\W_{\texttt{NLM}}$ is irreducible, we can conclude from the Perron-Frobenius theorem \cite{Horn_Johnson_2012} that $\fix(\W_{\texttt{NLM}}) = \Span(\one)$. Thus, we are done if we can show that $\A \one  \neq \ze$. Indeed, it is clear that $\A \one  \neq \ze$ for inpainting and deblurring. Moreover, since $\A$ is a lowpass blur followed by downsampling in superresolution, we cannot have $\A \one  = \ze$. A similar argument also goes for $\W_{\texttt{SYM}}$.
\end{proof}

\color{black}
  \section{Convergence Analysis}
  \label{sec:analysis}

  It was noted in prior works that we can establish convergence of \texttt{FISTA} for different step size rules \cite{chambolle2015convergence,lorenz2015inertial}. This is also the case for \texttt{PnP-FISTA} and \texttt{RED-APG}. In particular, instead of  \texttt{PnP-FISTA}, we will work with  \Cref{alg:alpha-pnp-fista}, where we have replaced the  momentum parameter $(t_k-1)/t_{k+1}$ in \cref{tk} by $\alpha_k$.   Similarly, instead of  \texttt{RED-APG}, we will work with \Cref{alg:alpha-red-apg}. 

  \begin{algorithm}
    \caption{$\alpha$-\texttt{PnP-FISTA}}
    \label{alg:alpha-pnp-fista}
    \begin{algorithmic}[1]
      \STATE{\textbf{input}:  loss $f$,  denoiser $\W$, $\x_0$, $\{\alpha_k\}_{k \geqslant 1}$ and $\lambda > 0, \, \gamma > 0$.}
      \STATE{\textbf{set}: $ \y_1=\x_0$.}
      \FOR {$k \geqslant 1$}
      \STATE{$\x_{k} =\W  (\y_k - \gamma \nabla \! f(\y_k))$.} \label{pnp:ytox}
      \STATE{$\y_{k+1} = \x_{k} +\alpha_{k} \, (\x_{k}-\x_{k-1})$.} \label{pnp:xtoy}
      \ENDFOR
    \end{algorithmic}
  \end{algorithm}

  \begin{algorithm}
    \caption{$\alpha$-\texttt{RED-APG}}
    \label{alg:alpha-red-apg}
    \begin{algorithmic}[1]
      \STATE{\textbf{input}:  loss $f$,  denoiser $\W$, $\v_0$,  $\{\alpha_k\}_{k \geqslant 1}$ and $\lambda > 0, \, L \geqslant 1$.}
      \STATE{\textbf{set}: $\theta=1/L$.}
      \FOR {$k \geqslant 1$}
      \STATE{$\x_{k} =\prox_{(\lambda L)^{-1} \, f } (\v_{k-1})$.} \label{red:vtox}
      \STATE{$\y_{k} = \x_{k} + \alpha_{k} \, (\x_{k}-\x_{k-1})$.}  \label{red:xtoy}
      \STATE{$\v_{k} = \theta \, \W(\y_k)+ (1-\theta) \, \y_k$.}  \label{red:ytov}
      \ENDFOR
    \end{algorithmic}
  \end{algorithm}

  We will analyze the iterate convergence of \Cref{alg:alpha-pnp-fista} and \Cref{alg:alpha-red-apg}. The key idea is to express the evolution of the iterates as a linear dynamical system in higher dimensions.

  \subsection{Preliminaries}  We say that a sequence $\{\x_k\} \subset \Re^n$ converges linearly to $\x^* \in \Re^n$ if there exists a norm $\| \cdot \|$ on $\Re^n$, $K \geqslant 1, \, C >0$, and $\beta \in [0,1)$ such that
  $\|\x_k  - \x^* \| \leqslant C \beta^k$ for all $k \geqslant K$. This asymptotic form of convergence is called ``local'' linear convergence in \cite{johnstone2015lyapunov,tao2016local}. However, since we use the term ``global convergence'' to mean that the sequence converges for any initialization, we will not use ``local'' and refer to this as linear convergence. The notion of linear convergence does not depend on the choice of norm.

We make the following elementary observation regarding the convergence of a time-varying linear dynamical system. 

  \begin{lemma}
    \label{lemma1}
    Suppose we are given $\s \in \Re^d$ and a sequence of linear operators $\{\R_k\}_{k \geqslant 1} \subset \Re^{d \times d}$. For all $k \geqslant 1$, define the affine map $\T_k : \Re^d \to \Re^d, \, \T_k(\z) = \R_k \z + \s$. Suppose there exists $\R_{\infty} \in \Re^{d \times d}$ such that
    
    \begin{enumerate}
      \item $\rho(\R_{\infty}) < 1$,
      \item $\lim_{k \to \infty}  \R_k = \R_{\infty}$, and
      \item $\bigcap_{k=1}^{\infty} \, \fix (\T_k) \neq \O$.
    \end{enumerate}
    
    Then there exists an unique $\z^* \in \Re^d$ such that, for any $\z_0 \in \Re^d$, the sequence $\{\z_k\}_{k \geqslant 0}$ generated via $\z_{k+1} = \T_k (\z_k)$
    converges linearly to $\z^*$.
  \end{lemma}

  \begin{proof}
   Fix $\beta <1 $ such that $\rho(\R_{\infty}) < \beta$. {By \cite[Lemma~5.6.10]{Horn_Johnson_2012}, there exists a nonsingular matrix $\Q\in\Re^{d\times d}$ such that, if we define the norm $\| \x \|_{\beta}:= \| \Q\x \|_1$ on $\Re^d$, then}
    \begin{equation*}
      \rho(\R_{\infty}) \leqslant \lvvvert  \R_{\infty} \rvvvert _{\beta} < \beta,
    \end{equation*}
    where  $\lvvvert \cdot \rvvvert_{\beta}$ is the operator norm induced by $\| \cdot \|_{\beta}$. 
   
    Now, since $\lvvvert \R_k - \R_{\infty} \rvvvert_{\beta} \to 0$, we can find $K \geqslant 1$ such that $\lvvvert \R_k \rvvvert_{\beta} \leqslant  \beta$ for all $k \geqslant K$.

    Let $\z^* \in \bigcap_{k=1}^{\infty}  \fix (\T_k)$, i.e., $\z^*=\T_k (\z^*)$ for all $k \geqslant 1$. We can write
    \begin{equation*}
      \z_{k+1} - \z^* = \T_k (\z_{k}) - \T_k(\z^*) = \R_k(\z_{k} - \z^*).
    \end{equation*}
    Hence, for all $k \geqslant K$, we have
    \begin{equation*}
      \| \z_{k+1} - \z^* \|_{\beta} \leqslant \lvvvert  \R_k  \rvvvert _{\beta} \, \| \z_{k} - \z^*  \|_{\beta} \leqslant \beta \| \z_{k} - \z^*  \|_{\beta}.
    \end{equation*}
    Iterating this, we get
    \begin{equation*}
      \forall \, k \geqslant K: \qquad \| \z_{k+1} - \z^* \|_{\beta} \leqslant  \| \z_K - \z^*  \|_{\beta} \cdot \beta^{k+1-K} = C \beta^{k+1},
    \end{equation*}
    where $C$ does not depend on $k$. Thus, we have shown that $\{\z_k\}$ converges linearly to $\z^*$ for any initialization $\z_0$.

    Finally, note that $\bigcap_{k=1}^{\infty}  \fix (\T_k)$ must be a singleton. This is because $\z^*=\R_k \z^*+\s$
    for all $k \geqslant 1$. Consequently, letting $k \to \infty$, we get
    \begin{equation}
      \label{unique}
      (\I_d-\R_{\infty} )\z^*=\s
    \end{equation}
    As $\rho (\R_{\infty} ) <1$, $\I_d-\R_{\infty} $ is nonsingular, and hence  \cref{unique} has at most one solution.
  \end{proof}

We record a result that is required later in the analysis.

  \begin{proposition}
    \label{rholt1}
    Let $\P \in \Re^{n \times n}$  be such that $\sigma(\P) \subset [0,1)$. If $\R \in \Re^{2n \times 2n}$ is defined to be
    \begin{equation*}
      \R=
      \begin{pmatrix}
        2 \P & - \P\\
        \I_n & \phantom{-} \ze_n
      \end{pmatrix},
    \end{equation*}
    then $\rho(\R) < 1$.
  \end{proposition}

  \begin{proof}
    Let $\lambda \in \mathbb{C}$ be a nonzero eigenvalue of $\R$. By definition, there exists $\z_1, \z_2 \in \Re^n$ such that $(\z_1,\z_2) \neq \ze$, and
    \begin{subequations}
      \begin{align}
        2\P \z_1 - \P \z_2 &= \lambda \z_1, \label{e1} \\
        \z_1&= \lambda \z_2. \label{e2}
      \end{align}
    \end{subequations}

    We claim that $\lambda \neq 1/2$ and $\z_1=\ze $. Indeed, if $\lambda = 1/2$, we would obtain $\z_1=\z_2=\ze $, contrary to our assumption. Similarly, $\z_1=\ze $ would imply (as $\lambda \neq 0$) that $\z=\ze$.

    Substituting \cref{e2} in \cref{e1}, we have
    \begin{equation}
      \label{eqn:eigvec}
      \P \z_1 = \mu \z_1, \qquad \mu:= \frac{\lambda^2}{2 \lambda -1}.
    \end{equation}
    Since $\z_1 \neq \ze$, it follows from \cref{eqn:eigvec} that $\mu \in \sigma(\P)$. Moreover, we have 
    \begin{equation*}
      \lambda^2 - 2\mu \, \lambda + \mu= 0.
    \end{equation*}
    As $\mu \in [0,1)$ by assumption, we have $\lambda=\mu \pm i \sqrt{\mu - \mu^2}$. In either case, $\lvert \lambda \rvert = \sqrt{\smash[b] { \mu}} < 1$.
  \end{proof}

We next prove a result about the spectrum of a particular operator involving the forward model $\A$ and the denoising operator $\W$.

  \begin{lemma}
    \label{lemma2}
    For $\W \in \Sy^n$ and $\A \in  \Re^{m \times n}$ satisfying \cref{RNP}, define
    \begin{equation}
      \label{defP}
      \P_{\gamma} = \W (\I_n - \gamma \A^\top \! \A) \quad \mbox{and} \quad \P_{\mu, \theta} =  (\I_n +\mu \A^\top \! \A)^{-1} \, (\theta \, \W + (1-\theta) \I).
    \end{equation}
    Then $\sigma(\P_{\gamma} ) \subset [0,1)$ for all $\gamma \in (0,1/\lambda_{\max}(\A^\top \! \A))$ and $\sigma(\P_{\mu,\theta}) \subset [0,1)$ for all $\theta \in [0,1], \, \mu>0$.
  \end{lemma}

  We need the following observation to prove \Cref{lemma2}. 

  \begin{proposition}
    \label{prop2}
    Let $\M \in \Sy^n$ and $\sigma(\M) \subset [0,1]$. Then $\|\M\x\|=\|\x\|$ if and only if $\M\x=\x$.
  \end{proposition}

  \begin{proof}[Proof of \Cref{lemma2}]
    Defining {$\G_{\gamma} = \I_n -\gamma\A^\top\!\A \in \Sy^n$} we have that  $\sigma(\G_{\gamma}) \subset (0,1]$  for $0 < \gamma < 1/\lambda_{\max}(\A^\top \! \A)$. In particular, as $\G_{\gamma}$ is nonsingular, we can write
    \begin{equation*}
      \label{sim}
      \P_{\gamma} = \G_{\gamma}^{-\frac{1}{2}} \, \big(\G_{\gamma}^{\frac{1}{2}}\W \G_{\gamma}^{\frac{1}{2}} \big) \, \G_{\gamma}^{\frac{1}{2}} \sim \G_{\gamma}^{\frac{1}{2}}\W \G_{\gamma}^{\frac{1}{2}}.
    \end{equation*}
    Thus, we have $\sigma(\P_{\gamma}) \subset [0,\infty)$ since $\W$ is positive semidefinite. Using \Cref{prop2}, we will establish the stronger result that $\rho(\P_{\gamma}) < 1$. In fact, we know that\footnote{This part is similar to the proof of \cite[Lemma~1]{ACK2023-contractivity}; however, the present argument is simpler as our assumption on the spectrum is weaker. Moreover, the analysis is necessary for the other operator $ \P_{\mu, \theta}$.}
    \begin{equation*}
      \rho(\P_{\gamma}) \leqslant \max_{ \|\x \|_2=1} \, \|\P_{\gamma}\x \|_2.
    \end{equation*}
    Thus, it suffices to show that $\|\P_{\gamma}\x \|_2  <1 $ for all $\|\x\|_2=1$.

    Note that $\|\W\|_2 \leqslant 1$ and $\|\G_{\gamma}\|_2 \leqslant 1$. Let $\x \in \Re^n$ such that $\|\x\|_2=1$. We have two possibilities: (i) $ \x \in \ker(\A) $, and (ii) $ \x  \notin  \ker(\A)$. For (i), we must have $\x \notin \fix (\W)$, since $\ker(\A) \, \cap \, \fix(\W) = \{ \ze\}$. Hence, from \Cref{prop2}, we get $\|\W\x\|_2 < \|\x\|_2$. Moreover, since $\G_{\gamma} \x=\x$,  
    \begin{equation*}
      \|\P_{\gamma} \x\|_2 = \|\W (\G_{\gamma} \x)\|_2 \leqslant \|\W  \x\|_2 < \|\x\|_2.
    \end{equation*}
    For (ii), we have $\G_{\gamma} \x \neq \x$; hence $\|\G_{\gamma} \x \|_2 < \|\x\|_2$ by \Cref{prop2}. Again, 
    \begin{equation*}
      \|\P_{\gamma} \x\|_2 = \|\W (\G_{\gamma} \x )\|_2 \leqslant  \|\G_{\gamma} \x\|_2< \|\x\|_2.
    \end{equation*}
    Thus, we have shown that $\sigma(\P_{\gamma} ) \subset [0,1)$.

    The other result $\sigma(\P_{\mu,\theta}) \subset [0,1)$ follows similarly. In particular, we can write $\P_{\mu,\theta} = \H \W_{\theta}$, where $\H:=(\I+\mu \A^\top \! \A)^{-1}$ and $\W_\theta :=\theta \, \W + (1-\theta) \I$. Since $\mu >0$ and $\theta \in [0,1]$, $\H$ and $\W_{\theta}$ are symmetric with eigenvalues in $[0,1]$. Moreover, it is clear that $\|\W_{\theta}\|_2 \leqslant 1$ and $\|\H \|_2 \leqslant 1$. Also, since $\fix (\W_{\theta})  = \fix (\W)$ and $ \fix(\H) = \ker(\A)$, we have
    \begin{equation*}
      \fix (\W_{\theta}) \, \cap \, \fix(\H)=\{\ze\}.
    \end{equation*}
    The rest of the proof is identical to that for $\P_{\gamma}$.
  \end{proof}

  \subsection{Linear convergence} We now discuss the main result of the paper. The idea is to express \Cref{alg:alpha-pnp-fista} and \Cref{alg:alpha-red-apg} as a dynamical system in $\Re^{2n}$. The transition operator is time-varying but is shown to converge to an operator with a spectral radius strictly less than $1$.

  \begin{theorem}
    \label{thm:linconvg}
    Let \cref{qloss} be the loss in \Cref{alg:alpha-pnp-fista} and \Cref{alg:alpha-red-apg}, $ \lim_{k \to \infty} \, \alpha_k =1$, and
 that  $\W$ and $\A$ satisfy \cref{RNP}. Then the iterates of \Cref{alg:alpha-pnp-fista} exhibit global linear convergence for any $\gamma \in (0,1/\lambda_{\max}(\A^\top \! \A))$. On the other hand, the iterates of \Cref{alg:alpha-red-apg}  exhibit global linear convergence for any $L \geqslant 1$.
  \end{theorem}

  \begin{proof}
    The gradient of \cref{qloss} is $\nabla f (\x) = \A^\top \! (\A\x-\b)$. Therefore, we can write step \ref{pnp:ytox} of \Cref{alg:alpha-pnp-fista} as
    $\x_k = \P_{\gamma} \, \y_k + \q$, where $\P_{\gamma}$ is defined in \cref{defP} and $\q:= \gamma \W\A^\top \! \b$. Furthermore, from step \ref{pnp:xtoy}, we have
    \begin{equation}
      \label{recurx}
      \x_k = (1+\alpha_k)  \, \P_{\gamma} \x_{k-1} -\alpha_{k-1} \P\x_{k-2}+ \q.
    \end{equation}

    For $k \geqslant 1$, define the state space variables
    \begin{equation}
      \label{defz}
      \z_k  =
      \begin{pmatrix}
        \x_{k} \\
        \x_{k-1}
      \end{pmatrix} \in \Re^{2n}.
    \end{equation}
    Since $\ker(\A) \, \cap \, \fix(\W) = \{ \ze\}$ and $\gamma \in (0,1/\lambda_{\max}(\A^\top \! \A))$, we know from \Cref{lemma2} that $\rho(\P_{\gamma}) < 1$. In particular, this means $\I_n - \P_{\gamma}$ is nonsingular, and we define
    \begin{equation}
      \label{defz*}
      \x^* :=(\I_n - \P_{\gamma})^{-1} \q, \qquad  \quad  \z^*  :=
      \begin{pmatrix}
        \x^* \\
        \x^*
      \end{pmatrix} \in \Re^{2n}.
    \end{equation}

    We now use \Cref{lemma1} to show  that $\{\z_k\}_{k \geqslant 1}$ converges linearly to $\z^*$. From \cref{recurx}, we can write $\z_k = \T_k (\z_{k-1})$, where
    \begin{equation*}
      \R_k :=
      \begin{pmatrix}
        (1+\alpha_k) \P & - \alpha_k \P\\
        \I_n & \phantom{-} \ze_n
      \end{pmatrix},
      \quad
      \s :=
      \begin{pmatrix}
        \q \\
        \ze
      \end{pmatrix},
      \quad
      \T_k(\z) := \R_k \z + \s.
    \end{equation*}
    The setup is now similar to \Cref{lemma1}. We define
    \begin{equation*}
      \R_{\infty}:=
      \begin{pmatrix}
        2 \P_{\gamma} & - \P_{\gamma}\\
        \I_n & \phantom{-} \ze_n
      \end{pmatrix},
    \end{equation*}
    and verify the assumptions in \Cref{lemma1}. Since $\alpha_k \to 1$, it is evident that $\lim_{k \to \infty}  \R_k = \R_{\infty}$. On the other hand, it follows from \Cref{lemma2} that $\sigma(\P_{\gamma}) \subset [0,1)$. Hence, from \Cref{rholt1}, we have $\rho(\R_{\infty})<1$. Finally, we can check from \cref{defz*} that
    $\z^* \in \fix (\T_k)$ for all $k \geqslant 1$. Therefore, we can conclude from \Cref{lemma1} that there exist $\beta \in [0,1)$, constant $C>0$, norm $\| \cdot \|$ on $\R^{2n}$, and $K \geqslant 1$ such that
    \begin{equation}
      \label{linconv}
      \forall \, k \geqslant K: \quad \|\z_k - \z^* \| \leqslant C \, \beta^{k}.
    \end{equation}
    Since $\|\x_k - \x^*\| \leqslant c \, \|\z_k - \z^* \| $ for some constant $c >0$ (see definition \cref{defz}), it is immediate from \cref{linconv} that $\{\x_k\}_{k \geqslant 1}$ converges linearly to $\x^*$.

    The analysis for \Cref{alg:alpha-red-apg} is similar. For the loss in \cref{qloss}, we have
    \begin{equation*}
      \prox_{(\lambda L)^{-1} \, f } (\y) = \H_{\mu} \y + \r,
    \end{equation*}
    where $\mu:=(\lambda L)^{-1}, \, \H_{\mu}:= (\I_n+\mu \A^\top\! \A)^{-1}$, and $\r:= \mu \, \H_{\mu} \A^\top\! \b$. Thus, we can write step \ref{red:vtox} of \Cref{alg:alpha-red-apg} as $\x_{k} = \H_{\mu} \v_{k-1} + \r$. Combining this with steps \ref{red:xtoy} and \ref{red:ytov}, we have
    \begin{equation*}
      \x_k = (1+\alpha_{k-1})\,  \P_{\mu, \theta} \, \x_{k-1} - \alpha_{k-1} \P_{\mu, \theta} \, \x_{k-2} + \r,
    \end{equation*}
    where
    \begin{equation*}
      \P_{\mu, \theta} := \H_{\mu} \W_{\theta} \quad \mbox{and} \quad \W_{\theta}:= \theta \W + (1-\theta) \I_n.
    \end{equation*}

    Beyond this point, the analysis proceeds as before. Namely, we move from the iterates $\{\x_k\}$ to the state space sequence $\{\z_k\}$, establish linear convergence of $\{\z_k\}$ using \Cref{rholt1} and \Cref{lemma1}, and finally deduce linear convergence of $\{\x_k\}$.
  \end{proof}
  \begin{corollary}
    \label{corr1}
    Suppose \texttt{PnP-FISTA} or \texttt{RED-APG} is used for inpainting, deblurring, or superresolution, where the loss is \cref{qloss} and the {denoiser is \cref{DSG-NLM}.} Then the iterates of \texttt{PnP-FISTA} (resp.~\texttt{RED-APG}) converge globally and linearly to a unique reconstruction for any $0 < \gamma < 1/\lambda_{\max}(\A^\top \! \A)$ (resp.~$L \geqslant 1$).
  \end{corollary}
  
\begin{proof}
This follows from \Cref{thm:linconvg}. Indeed, we have $\alpha_k = (t_k-1)/t_{k+1}$ for \texttt{PnP-FISTA} and \texttt{RED-APG} which converges to $1$.
Hence, we have $\alpha_k \to 1$ for \texttt{PnP-FISTA} and \texttt{RED-APG}. The rest of the argument follows from \cref{verification-of-rnp} and \cref{thm:linconvg}.
  \end{proof}

  \subsection{Extension to nonsymmetric $\W$}
  \label{subsection:nonsymm}
  The convergence analysis for \texttt{PnP-FISTA}  can be extended to the nonsymmetric kernel denoiser $\W_\texttt{NLM}$ defined in \cref{NLM}. However, as pointed out in \cite{gavaskarPlugandPlayRegularizationUsing2021a}, we need to work with an inner product different from the standard inner product on $\Re^n$. More precisely, we need to run \texttt{PnP-FISTA} in the Euclidean space $(\Re^n, \langle \cdot, \cdot \rangle_{\D})$, where the inner-product $\langle \cdot, \cdot \rangle_{\D}$ is given by\footnote{Note that $\W_\texttt{NLM}$ is self-adjoint with respect to \cref{ip}. Moreover, as observed in \cite{gavaskarPlugandPlayRegularizationUsing2021a}, $\W_\texttt{NLM}$ can be expressed as the proximal operator (w.r.t. the norm induced by \cref{ip}) of a closed, proper, and convex function.}
    \begin{equation}
    \label{ip}
    \langle \x, \y \rangle_{\D} := \x^\top \D \y.
  \end{equation}
  We note that \cref{ip} is a valid inner product since $\D$ is positive definite by construction.
  
  Accordingly, step \ref{pnp:ytox} in \texttt{PnP-FISTA} (and \Cref{alg:alpha-pnp-fista}) must be changed to
  \begin{equation}
    \label{newupdate}
    \x_{k} =\W_\texttt{NLM}  \big(\y_k - \gamma \nabla_{\D}  f(\y_k) \big),
  \end{equation}
  where $\nabla_{\D}  f$ is the gradient of $f$ with respect to the inner-product \cref{ip}. In fact, we have $\nabla_{\D}  f(\y) = \D^{-1} \A^\top \! (\A\y-\b)$. Hence, we can write \cref{newupdate} as
  \begin{equation*}
    \x_{k} =\W_\texttt{NLM}  \big(\y_k - \gamma \D^{-1} \A^\top \! (\A\y_k-\b) \big) = \P_{\gamma} \y_k +  \q,
  \end{equation*}
  where
  \begin{equation}
    \label{newP}
    \P_{\gamma} := \W_\texttt{NLM} (\I_n - \gamma \D^{-1} \A^\top \! \A) \quad \mbox{and} \quad \q:= \gamma \W_\texttt{NLM}\D^{-1} \A^\top \!  \b.
  \end{equation}
  The other updates in  \texttt{PnP-FISTA} stay unchanged. Following \cite{gavaskarPlugandPlayRegularizationUsing2021a}, we will refer to this modified algorithm as \texttt{scaled-PnP-FISTA}. Similar to \Cref{corr1}, we have the following result\footnote{A similar analysis holds for \texttt{RED-APG}, where we can establish linear convergence for a nonsymmetric kernel denoiser by incorporating \cref{ip} and the corresponding norm in \Cref{alg:alpha-red-apg}.}.

  \begin{corollary}
    \label{corr2}
    Suppose we use \texttt{scaled-PnP-FISTA}  for inpainting, deblurring, or superresolution, where the loss is \cref{qloss} and the denoiser is \cref{NLM}. Then the iterates exhibit global linear convergence for any {$0 < \gamma < 1/\lambda_{\max}(\D^{-1/2}\A^\top \! \A\D^{-1/2})$}.
  \end{corollary}

 {We remark that since $\D_{ii}\geqslant 1$, we have $\lambda_{\max}(\D^{-1/2}\A^\top\!\A\D^{-1/2}) \leqslant \lambda_{\max}(\A^\top\!\A)$. Thus, we can also use the interval $(0,1/\lambda_{\max}(\A^\top\!\A))$ in \Cref{corr2}, which does not depend on $\D$.}

  \begin{proof}
   We are done if we can show that $\sigma(\P_{\gamma}) \subset [0,1)$, where $\P_{\gamma}$ is as defined in \cref{newP}. The rest of the analysis is identical to \Cref{corr1}.

    To verify $\sigma(\P_{\gamma}) \subset [0,1)$, we use \cref{NLM} to write $\P_{\gamma} \sim \D^{\frac{1}{2}} \, \P_{\gamma} \, \D^{-\frac{1}{2}} = \W_s\G_{\gamma}$, where
    \begin{equation*}
      \W_s:= \D^{-\frac{1}{2}}\K \D^{-\frac{1}{2}} \quad \mbox{and} \quad \G_{\gamma} := \I_n - \gamma  \D^{-\frac{1}{2}} \A^\top \! \A  \D^{-\frac{1}{2}}
    \end{equation*}
    In particular, $\sigma(\P_{\gamma}) = \sigma(\W_s\G_{\gamma})$ and $\W_s, \, \G_{\gamma} \in \Sy^n$. We are now in the setup of  \Cref{lemma2}.
    {Since $\W_s$ is similar to $\W_{\texttt{NLM}}$ and $\sigma(\W_\texttt{NLM}) \subset [0,1]$ (\Cref{propertiesW}), we have $\sigma(\W_s) \subset [0,1]$.}
    On the other hand, it is not difficult to verify that $\sigma(\G_{\gamma}) \subset [0,1]$ if {$0 < \gamma < 1/\lambda_{\max}(\D^{-1/2}\A^\top \! \A\D^{-1/2})$}. Furthermore, 
    \begin{equation*}
      \fix(\W_s) \, \cap \, \fix(\G_{\gamma}) = \big(\D^{\frac{1}{2}}\, \Span(\one)\big) \,  \cap \, \big( \D^{\frac{1}{2}} \ker (\A)\big) = \{\ze \},
    \end{equation*}
    since $ \fix(\W_s) = \D^{\frac{1}{2}}  \fix(\W_{\texttt{NLM}})$  and $\fix(\W_\texttt{NLM})=\Span(\one)$ by \Cref{propertiesW}.
    Similar to \Cref{lemma2}, we can now easily verify that $\sigma(\W_s\G_{\gamma}) \subset [0,1)$.
  \end{proof}

\section{Numerical results}
\label{sec:exp}

{We present results\footnote{The code is available here: \url{https://github.com/arghyasinha/PnP-FISTA}.} for inpainting, deblurring, and superresolution\footnote{The test images are from the Set12 database: \url{https://github.com/cszn/DnCNN/tree/master/testsets/Set12}} to provide numerical evidence supporting the theoretical results.
In particular, we validate the convergence theorems, confirm the uniqueness of the reconstruction for different initializations, and test whether $\rho(\R_{\infty}) < 1$.}

{For deblurring, we applied an isotropic Gaussian blur of $25\times 25$ size and standard deviation $1.6$. Additionally, we introduced white Gaussian noise of standard deviation $\sigma=0.03$.} For inpainting, we randomly sampled $30\%$ pixels. In all the experiments, we have used the symmetric \texttt{DSG-NLM} denoiser~\cite{sreehari2016plug}. The deblurring and inpainting results are shown in \Cref{fig:deblurring} and \Cref{fig:inpainting}. We demonstrate in \Cref{fig:inpainting} that the \texttt{PnP-FISTA} iterates converge to the same reconstruction independent of the initialization (\Cref{thm:linconvg}). {
The results for a  $2\times$-superresolution experiment are shown in \Cref{fig:SR}, where we have compared the reconstruction from \texttt{DSG-NLM} with that obtained using trained denoisers such as \texttt{DnCNN} \cite{dncnn} and \texttt{DRUNet} \cite{zhang2021plug}. We have used the Lipschitz-constrained \texttt{DnCNN}\footnote{\url{https://deepinv.github.io/deepinv/deepinv.denoisers.html\#pretrained-weights}} denoiser from \cite{pesquet_learning_2021}, which comes with a convergence guarantee for \cref{eq:pnpista}. The \texttt{DRUNet} result is superior to \texttt{DnCNN} \cite{pesquet_learning_2021} and \texttt{DSG-NLM}. {However, we note that \texttt{DRUNet} does not come with convergence guarantees for \texttt{PnP} and \texttt{RED} (see \Cref{fig:divergence}). While \texttt{DnCNN} achieves a higher PSNR than \texttt{DSG-NLM}, it introduces visual artifacts, possibly due to the low noise levels used during its training \cite{pesquet_learning_2021}. Overall, we observe a tradeoff between the regularization capacity of a denoiser and its convergence properties, a challenge that is widely recognized within the community \cite{goujon_learning_2024,hurault2022gradient,nair2024averaged,terris2024equivariant}.}}

\begin{figure}[t]
\centering
\subfloat[ground truth]{\label{fig:f}\includegraphics[width = 0.24\textwidth]{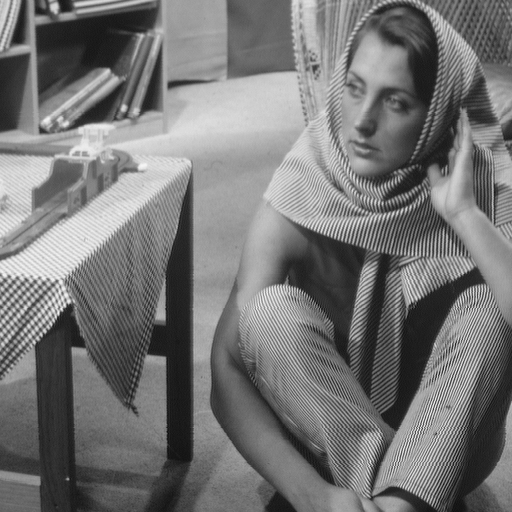}}
\hfill
\subfloat[observed ]{\label{fig:e}\includegraphics[width = 0.24\textwidth]{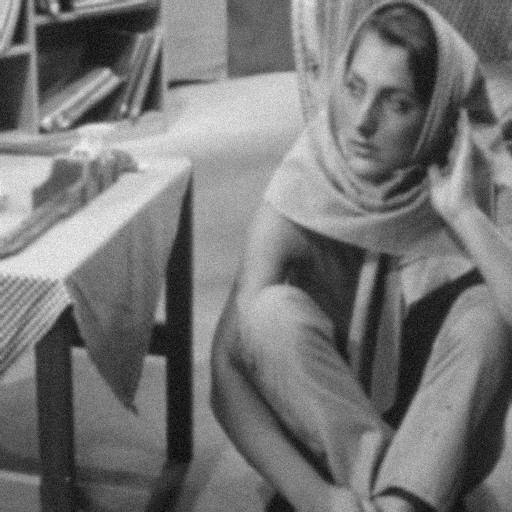}}
\hfill
\subfloat[\texttt{PnP-FISTA}]{\label{fig:g}\includegraphics[width = 0.24\textwidth]{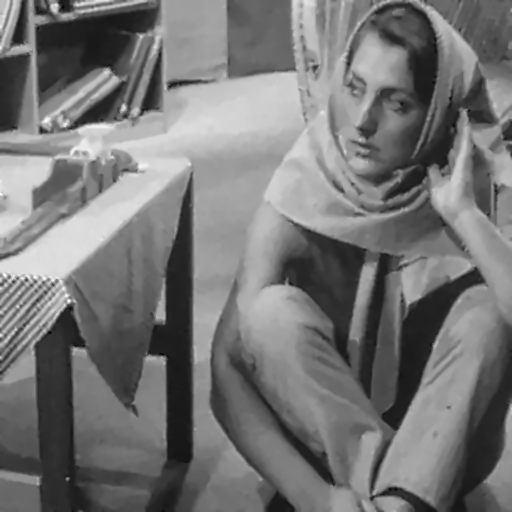}}
\hfill
\subfloat[\texttt{RED-APG}]{\label{fig:h}\includegraphics[width = 0.24\textwidth]{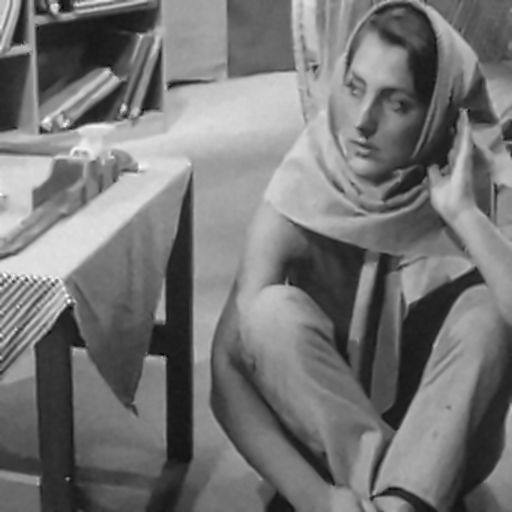}}
\caption{Deblurring using a symmetric kernel denoiser. The PSNR and SSIM values are (b) $23.07, \, 0.5317$, (c) $24.16, \, 0.7053$, and (d) $24.15, \, 0.7025$. We have used $\gamma=0.9$ for \texttt{PnP-FISTA} and $\lambda=1,L=2$ for \texttt{RED-APG} and (b) observed image as the guide image for the denoiser.}
\label{fig:deblurring}
\end{figure}

We next compute the spectral radius $\rho(\R_\infty)$ using the power method \cite{matrixcomputations}. We recall that  $\rho(\R_\infty)$ is a lower bound on the convergence rate $\beta$ in \cref{linconv}. As shown in \Cref{tab:1},  $\rho(\R_\infty)$ is very close to but strictly less than one for inpainting and {deblurring} (\Cref{rholt1})

 Finally, we study the empirical convergence of \Cref{alg:alpha-pnp-fista} and \Cref{alg:alpha-red-apg} for different choices of the momentum parameter $\alpha_k$ (\Cref{thm:linconvg}). This is done for the inpainting experiment in \Cref{fig:inpainting}.  Along with the $\{\alpha_k\}$ sequence in Beck~\cite{beck_fast_2009} and Chambolle~\cite{chambolle2015convergence}, we have used $\alpha_k =1 - 1/\log (k+1)$ and $\alpha_k = 1- (0.5)^k$ that have the property $\alpha_k \to 1$. We plot the distance of the iterates from the unique limit point in \Cref{fig:plot}. Notably, observe that the iterates of \Cref{alg:alpha-pnp-fista} can be slower to converge compared to the non-accelerated variants \texttt{PnP-ISTA} and \texttt{RED-PG} (\Cref{alg:alpha-pnp-fista} and \Cref{alg:alpha-red-apg}), depending on the choice of $\{\alpha_k\}$. This is consistent with the findings in \cite{odonoghue_adaptive_2015,tao2016local}.

\begin{figure}
    \centering
    \includegraphics[width=0.9\linewidth]{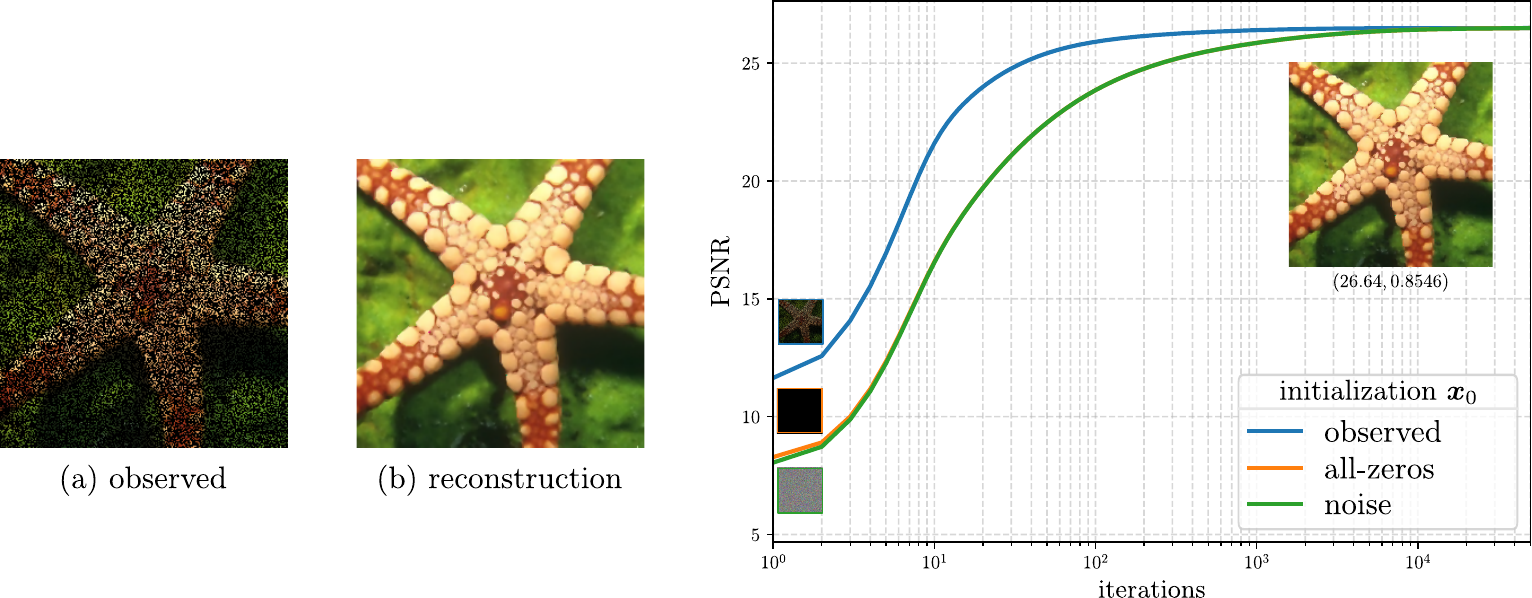}
    \caption{Global convergence of \texttt{PnP-FISTA}. Inpainting results with different initialization $\x_0$ (see the legend) converges to the same reconstruction (b) having PSNR $26.64$ and SSIM $0.8546$. We have used \texttt{PnP-FISTA} with $\gamma=0.9$, and median filtered the observed image to get the guide image. 
    }
    \label{fig:inpainting}
\end{figure}

\begin{figure}[tbhp]
    \subfloat[ground truth]{\includegraphics[width=0.18\textwidth]{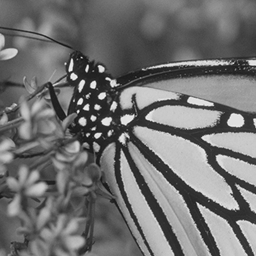}}\hfill
    \subfloat[observed]{\includegraphics[width=0.18\textwidth]{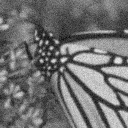}}\hfill
     \subfloat[\centering \texttt{DSG-NLM}]{\includegraphics[width=0.18\textwidth]{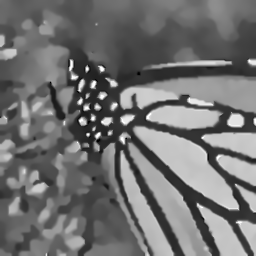}}\hfill
     \subfloat[\centering \texttt{DnCNN}]{\includegraphics[width=0.18\textwidth]{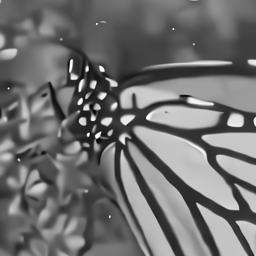}}\hfill
     \subfloat[\centering \texttt{DRUNet}]{\includegraphics[width=0.18\textwidth]{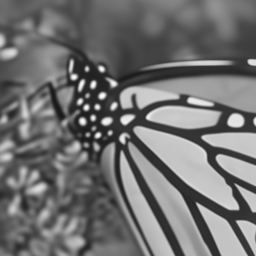}}
    \caption{Comparison of superresolution results with different denoisers. The forward operator is $\A = \S\B$, where $\B$ is $25\times 25$ Gaussian blur with standard deviation $1.6$ and $\S$ is $2\times$-downsampling. We also added white Gaussian noise of strength $\sigma = 0.04$ to the measurement. The PSNR and SSIM values are (c) $25.27, \, 0.8201$, (d) $25.35, \, 0.8324$, and (e) $25.99, \,0.8445$. We used bicubic interpolation on the observed image to generate a guide image for \texttt{DSG-NLM}.}
    \label{fig:SR}
\end{figure}

\begin{figure}
\centering
\subfloat[$\alpha$-\texttt{PnP-FISTA}]{\includegraphics[width = 0.5\textwidth]{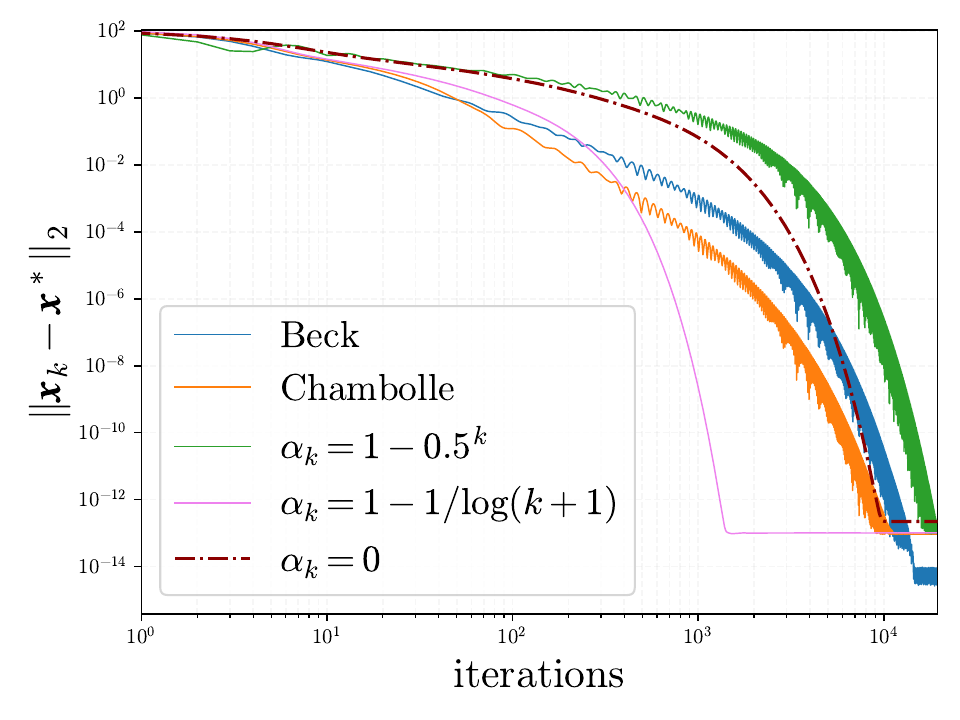}}
\subfloat[$\alpha$-\texttt{RED-APG}]{\includegraphics[width = 0.5\textwidth]{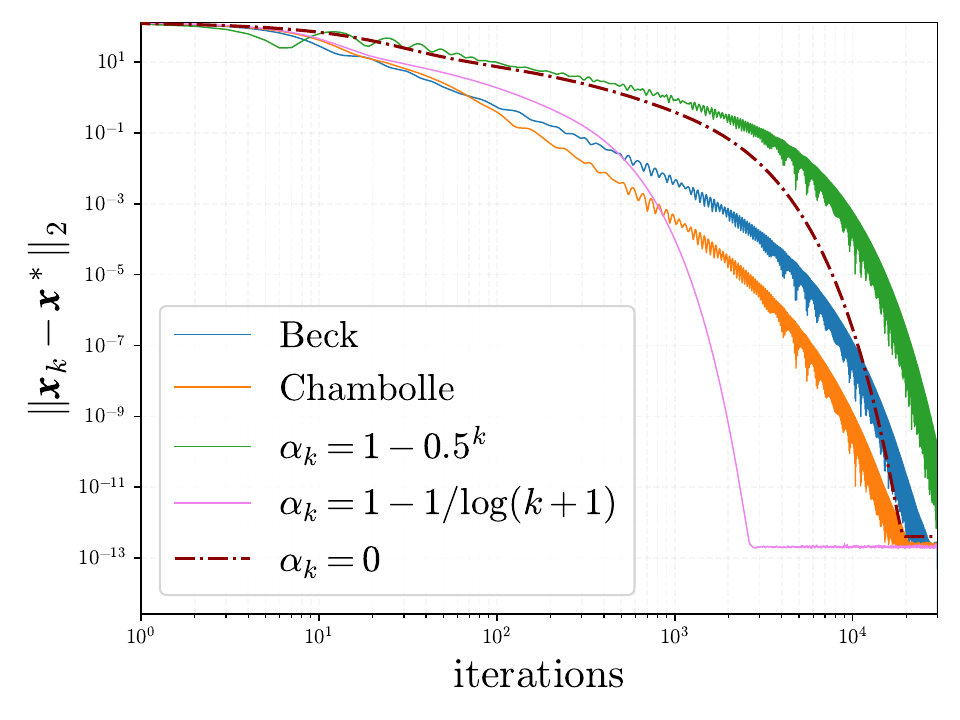}}
\caption{Iterate convergence for different $\{\alpha_k\}$ having the property $\alpha_k \to 1$. The case $\alpha_k=0$ correspond to  \texttt{PnP-ISTA} and \texttt{RED-PG}, the non-accelerated variants of \texttt{PnP-FISTA} and \texttt{RED-APG}. {The limit point $\x^*$ was approximated using $20000$ iterations of \texttt{PnP-FISTA} and  $30000$ iterations of \texttt{RED-APG} using the Beck sequence of $\{\alpha_k\}$. The parameters used are $\gamma = 0.9$ for \texttt{PnP-FISTA} and $\lambda = 1,L = 2$ for \texttt{RED-APG}.}}
\label{fig:plot}
\end{figure}

\begin{table}
{\footnotesize
  \caption{Values of $\rho({\R_\infty})$ for for different applications and parameter settings.}
  \label{tab:simpletable}
\begin{center}
\begin{tabular}{|c|c|c|c|c|} \hline
  & \multicolumn{2}{c|}{\texttt{PnP-FISTA}} &\multicolumn{2}{c|}{\texttt{RED-APG}}\\ \hline
$\gamma,1/L$ & Inpainting & Deblurring& Inpainting&Deblurring \\ \hline
0.10 & 0.993389 & 0.999920 & 0.996205  &  0.999914  \\
0.25 & 0.992201 & 0.999907 & 0.997254  &  0.999905  \\
0.50 & 0.991614 & 0.999893 & 0.994531  &  0.999893 \\
0.75 & 0.991415 & 0.999884 & 0.991786  &  0.999884\\
0.90 & 0.991297 & 0.999880 & 0.990136  &  0.999880\\ \hline
\end{tabular}
\end{center}
}
\label{tab:1}
\end{table}

\clearpage
  \section{Conclusion}
  \label{sec:conc}
  We presented a concise analysis of the global linear convergence of \texttt{FISTA} in the context of regularizing linear inverse problems with symmetric denoisers. Additionally, we showed how the results can be extended to nonsymmetric denoisers by operating in an appropriate Euclidean space and using a scaled variant of the original algorithm. We restricted the scope to linear denoisers to keep the analysis tractable. However, data-driven linear denoisers are interesting in their own right in that they come with fast algorithms \cite{nair2019fast} and give good reconstructions for different applications. It would be interesting to see if the present analysis could shed some insights on the convergence theory of trained nonlinear denoisers that give state-of-the-art reconstructions.

  \bibliographystyle{siamplain}
  \bibliography{references}

\end{document}